\documentclass{article}

\usepackage{graphicx}
\usepackage{amsmath, amsfonts, graphicx, color, enumerate,amssymb,amsthm}
\usepackage{tikz}
\usepackage[round]{natbib}

\DeclareMathOperator{\Out}{Out}

\DeclareMathOperator{\Inv}{Inv}
\DeclareMathOperator{\Aut}{Aut}
\DeclareMathOperator{\Der}{Der}
\DeclareMathOperator{\Lie}{Lie}
\DeclareMathOperator{\Hom}{Hom}
\DeclareMathOperator{\diag}{diag}
\DeclareMathOperator{\id}{id}
\DeclareMathOperator{\ch}{char}

\DeclareMathOperator{\SL}{SL}
\DeclareMathOperator{\Sp}{Sp}

\DeclareMathOperator{\G}{G}
\DeclareMathOperator{\C}{C}
\DeclareMathOperator{\D}{D}
\DeclareMathOperator{\F}{F}
\DeclareMathOperator{\E}{E}
\DeclareMathOperator{\Spin}{Spin}
\DeclareMathOperator{\Mat}{Mat}

\DeclareMathOperator{\Tr}{Tr}
\DeclareMathOperator{\N}{N}
\DeclareMathOperator{\tr}{tr}

\DeclareMathOperator{\Sr}{Sr}
\DeclareMathOperator{\A}{A}

\DeclareMathOperator{\Lm}{L}
\DeclareMathOperator{\Lin}{Lin}
\DeclareMathOperator{\U}{U}
\DeclareMathOperator{\q}{q}
\DeclareMathOperator{\n}{n}
\DeclareMathOperator{\Her}{Her}

\newtheorem{thm}[subsection]{Theorem}
\newtheorem{prop}[subsection]{Proposition}
\newtheorem{lem}[subsection]{Lemma}
\newtheorem{cor}[subsection]{Corollary}

\newcommand{\I}{\mathcal{I}}

\title{Isomorphism classes of $k$-involutions of algebraic groups of type $\E_6$ \thanks{SAU research fund} }

\author{John Hutchens} 

\begin{document}

\maketitle

\begin{abstract}
Automorphisms of order $2$ are studied in order to understand generalized symmetric spaces.  The groups of type $\E_6$ we consider here can be realized as both the group of linear maps that leave a certain determinant invariant, and also as the identity component of the automorphism group of a class of structurable algebras known as Brown algebras.  We will classify the $k$-involutions of these groups of type $\E_6$ using aspects of both descriptions.
\end{abstract}

\section{Introduction}

Symmetric spaces of algebraic groups were first studied by Gantmacher in \citep{ga39} to classify real Lie groups.  The symmetric spaces of real and complex Lie algebras were classified by Berger in \citep{be57}.  This classification was later extended to exceptional Lie groups explicitly by Yokota in \citep{yo90}.  A similar classification of groups of centralizers of elements of order $2$ for finite groups of Lie type has been carried out in \citep{as76} and \citep{gls98}.

The current paper is a continuation of a series of papers addressing the classification of symmetric $k$-varieties of algebraic groups \citep{he00,he02}, and further a continuation within that series that focuses on exceptional algebraic groups \citep{hu12,hu14}.  We will consider an algebraic group $G$ defined over a field $k$, and an automorphism of order $2$, $\theta: G \to G$ such that $\theta\big(G(k)\big) = G(k)$ where $G(k)$ are the $k$-rational points of $G$.  Our ultimate goal is to classify spaces of the form $G(k)/H(k)$, where $H$ is the fixed point group of a $k$-involution of $G$ and $H(k)$ its $k$-rational points.

To this end we refer to \citep{he00} and A.G. Helminck's correspondence between the classification of such spaces and the following invariants

\begin{enumerate}[(1)]
\item classification of admissible involutions of $(X^*(T),X^*(S), \Phi(T), \Phi(S))$, where $T$ is a maximal torus in $G$, $S$ is a maximal $k$-split torus contained in $T$
\item classification of the $G(k)$-isomorphism classes of $k$-inner elements $a\in I_k(S_{\theta}^-)$
\item classification of the $G(k)$-isomorphism classes of $k$-involutions of the $k$-anisotropic kernel of $G$,
\end{enumerate} 

Our classification relies on the structure of Albert algebras usually denoted by $J$, and of another type of structurable algebra call a Brown algebra denoted by $B$.  The main result found in Theorem \ref{E6kinv} is a method of determining isomorphism classes of automorphisms of order $2$ of the groups $\Inv(J)$ or $\Aut^+(B)$ when $J$ is split, both of which are of type $\E_6$.  We show in Lemma \ref{fixJ} that most of the variation of the structure of the isomorphism classes of these automorphisms corresponds to the variation in isomorphism classes of certain quaternion algebras.  We also give explicit representatives of isomorphism classes of these automorphism of order $2$ in Theorem \ref{E6kinv2} for specific fields including algebraically closed fields, real numbers, finite fields, and $\mathfrak{p}$-adic fields.

In the following section we define terms specific to the definitions of the first two invariants.  Because we are only dealing with split cases of algebraic groups of type $\E_6$ we do not need to discuss the invariant of type (3).

\section{Preliminaries}

In the current paper we consider only split groups of type $\E_6$ over fields of characteristic not $2$ or $3$.  For characteristic $2$ and $3$ we refer the reader to \citep{as76,gls98}.  All of these groups can be realized as the isometry group of an Albert algebra, this is well known and can be found in \citep{sv00,fre59}.  In \citep{ga01} Garibaldi shows that groups of type $\E_6$ can also be realized as automorphism groups of a Brown algebra.  We always think of groups of type $\G_2$ and $\F_4$ as the automorphism groups of octonion algebras and Albert algebras respectively.

In this section we define a Brown algebra, and illustrate some characteristics of other algebras we utilize.  We denote the inner automorphisms induced by $g$ with $\mathcal{I}_g(x) = gxg^{-1}$, $g \in G$.  The group of characters and the root space associated to a torus $T$ will be denoted $X^*(T)$ and $\Phi(T)$ respectively.

When we say an algebraic group is \emph{$k$-split} we mean that it contains a maximal torus that is isomorphic to a product of multiplicative groups $G_m(k) \times G_m(k) \times \cdots \times G_m(k)$.  We call a torus, $S$, \emph{$\theta$-split} if $\theta(s)=s^{-1}$ for all $s\in S$, and we call a torus $(\theta,k)$-split if it is both $\theta$-split and $k$-split.  We denote the maximal $\theta$-split subtorus of $S$ by
	\[ S_{\theta}^- = \{ s \in S \ | \ \theta(s)=s^{-1} \}. \]
If $T$ is a maximal $k$-torus containing the subtorus $S$ we call
	\[ \theta \in \Aut(X^*(T),X^*(S),\Phi(T), \Phi(S) ), \]
an \emph{admissible involution} if there is an involution $\tilde{\theta} \in \Aut(G,T,S)$ such that the restriction of $\tilde{\theta}$ to $X^*(T)$ is $\theta$, $S_{\tilde{\theta}}^-$ is a maximal $(\tilde{\theta},k)$-split torus and $T_{\tilde{\theta}}^-$ is a maximal $\tilde{\theta}$-split torus.  An element $a \in G$ is a \emph{$k$-inner element} with respect to a $k$-involution $\theta$ if $a \in S_{\theta}^-$ and $\theta \mathcal{I}_a$ is a $k$-involution of $G$.

The groups we are interested in all depend, in one way or another, on a class of algebras known as composition algebras.  We say a quadratic form $\q: V \to k$ is \emph{nondegenerate} if its associate bilinear form
\[ \langle x, y \rangle = \q(x+y) - \q(x) - \q(y), \]
is nondegenerate, i.e. $V^{\perp} = \{0\}$. A \emph{composition algebra} is an algebra $C$ over a field $k$ equipped with a product and a nondegenerate quadratic form, $\q: C \to k$, such that $\q(xy) = \q(x)\q(y)$ for $x,y \in C$.  It was shown by Hurwitz \citep{hu22} that these can only occur in dimensions $1, 2, 4$, or $8$.  A four dimensional composition algebra is called a \emph{quaternion algebra}, and an eight dimensional composition algebra is called an \emph{octonion algebra}.  Composition algebras are equipped with an algebra involution denoted by $\bar{ \ }: C \to C$.  We define an \emph{algebra involution} to be an anti-homomorphism from the algebra to itself. 

A split quaternion algebra $D$ over $k$ is isomorphic to the algebra of $2 \times 2$ matrices over $k$.  The product is typical matrix multiplication and $\q_D$ is the typical $2 \times 2$ matrix determinant.  The algebra involution on the quaternion algebra is defined by
\[ \overline{\begin{bmatrix}
	\alpha_{11} & \alpha_{12} \\
	\alpha_{21} &  \alpha_{22}
\end{bmatrix}} = \begin{bmatrix}
	\ \alpha_{22} & -\alpha_{12} \\
	-\alpha_{21} & \ \alpha_{11}
\end{bmatrix}. \]  

We can construct a split octonion algebra from a split quaternion algebra through a doubling process referred to as the Cayley-Dickson process.  This octonion algebra is an ordered pair of $2 \times 2$ matrices over $k$.  The new product is given by
\[ (a_1,a_2)(b_1,b_2) = \left(a_1b_1 + \kappa\overline{b}_2 a_2, b_2a_1 + a_2 \overline{b}_1 \right), \]
where $\kappa \in G_m(k)$.  The new norm is given by $\q\big( (a_1,a_2) \big) = \q_D(a_1) -\kappa \q_D(a_2)$.

Our groups of type $\F_4$ are defined in terms of a class of Jordan algebras.  We let $\gamma = \diag(\gamma_1,\gamma_2,\gamma_3)$ where $\gamma_i \in G_m(k)$ for $1\leq i \leq 3$.  Let $\Her_3(C,\gamma)$ denote the set of $3 \times 3$ matrices over $C$ that are $\gamma$-Hermitian.  If $\mathcal{A}\otimes_k k_{sep} \cong \Her_3(C_{k_{sep}}, \gamma)$ where $k_{sep}$ is a separable closure of $k$, we call $\mathcal{A}$ an \emph{Albert algebra}.   

We can define a quadratic form $Q: J \to J$ on an Albert algebra for $j \in J$ by
\[ Q(j) = \frac{1}{2}\Tr(j^2), \]
where $\Tr$ is the trace form on $J$.  Let $u \in \Her_3(C,\gamma)$ such that $u^2 = u$, and $Q(u) = \frac{1}{2}$, then $u$ is called a \emph{primitive idempotent}.  The terminology is motivated by the fact that $Q(u)=\frac{1}{2}$ if and only if $u$ cannot be written as a linear combination of other idempotent elements.  We denote by $E_0$ the zero-space of multiplication by $u$ in $\Her_3(C,\gamma)$, that is orthogonal to the identity element $e$.  We are also interested in some quadratic Jordan algebras corresponding to a fixed primitive idempotent.  When $C$ is an octonion algebra we will denote the $11$ dimensional quadratic Jordan algebra by
\[ \beth = ku \oplus k(e-u) \oplus E_0, \]
where $e$ is the identity element in $J$.  For a thorough introduction to Jordan algebras see \citep{mcc05} or \citep{ja58}.

There is a cubic norm that acts as the determinant on $\Her_3(C,\gamma)$.  We will denote this norm by $\N$, and give a detailed definition later.  When $C$ is a split octonion algebra there is only one isomorphism class of $\beth$.  When $C$ is an octonion algebra and $J \cong \Her_3(C,\gamma)$ we call $J$ a \emph{reduced} Albert algebra, or equivalently if $J$ contains a nonzero idempotent.  So, a typical element of a reduced Albert algebra is of the form

\[ \begin{bmatrix}
\xi_1 & c & \gamma_1^{-1}\gamma_2 \overline{b} \\
\gamma_2^{-1}\gamma_1 \overline{c} & \xi_2 & a \\
b & \gamma_3^{-1}\gamma_2\overline{a} & \xi_3
\end{bmatrix}, \]
where $\xi_i \in k$, $\gamma_i \in G_m(k)$ and $a,b,c \in C$.  A dot will denote the Jordan product,
	\[ x\cdot y = \frac{1}{2}(x y + y x), \]
where juxtaposition denotes the typical matrix product.  The following result is well known and can be found in \citep{ja61,sv00} with the added hypothesis that the associated composition algebras are isomorphic.  This additional hypothesis is proved unnecessary in \citep{ga03}.  We will use the term \emph{$\mathcal{J}$-algebra} to refer to the family of Jordan algebras needed in the classification of $k$-involutions of $\Aut(J)$.  These include cubic Jordan algebras over composition algebras and the quadratic algebras of the same form as $\beth$ described above.

\begin{thm}
Two reduced $\mathcal{J}$-algebras $J$ and $J'$ isomorphic if and only if their quadratic forms are equivalent.
\end{thm}

A linear map of $k$-algebras $\phi: A \to A'$  is called a \emph{norm similarity} if $\N'(\phi(a)) = \lambda\N(a)$ with $\lambda \in G_m(k)$.  The value $\lambda$ is called the \emph{multiplier} of $\phi$.  If $\lambda = 1$ we say that $\phi$ is a \emph{norm isometry}.

We define a triple product on an Albert algebra $J$ by
\[ \{x,z,y\} = (x\cdot z)\cdot y + (y\cdot z)\cdot x - (x\cdot y)\cdot z. \]

We also think of these algebras in terms of the first Tits construction, and we review the necessary details.  To this end we define a sharped cubic form to be the triple $(\N,\#,1_X)$, where $X$ is a $k$-algebra, $\N$ is a norm form on $X$, and $1_X$ is the base point.  We define $\N(1_X) = 1$.  The sharp map $\#: X \to X$ is given by

\[ x^{\#} = x^2 - \Tr(x)x + \Sr(x)1_X, \]

and we define,

\[ x \# y = (x+y)^{\#} - x^{\#} - y^{\#}. \]

The forms $\Tr$ and $\Sr$ are referred to as the (linear) trace and quadratic trace respectively.  If the following identities hold

\begin{align}
\Tr(x^{\#}, y) &= \N(x,y) \\
\big(x^{\#}\big)^{\#} &= \N(x)x \label{scf2} \\
1_X \# x &= \Tr(x)1_X - x,
\end{align}
we call $(\N,\#,1_X)$ a \emph{sharped cubic form} on $X$.  We define a nondegenerate cubic form with base point to be a \emph{Jordan cubic form} if it satifies (\ref{scf2}).  It is well known that every Jordan cubic form is a sharped cubic form, see \citep{mcc05}.

From a sharped cubic form $(\N, \#, 1_X)$ we can construct a unital Jordan algebra, $J(\N, \#, 1_X)$, which has unit element $1_X$, and a $\U$-operator defined by
	\[ \U_xy = \Tr(x,y)x - x^{\#}\# y. \]
In our case we can use the following identity
\[ \U_x = 2\Lm_x^2 - \Lm_{x^2}. \]

	\begin{prop}
	Any sharped cubic form $(\N,\#,1_X)$ gives a unital Jordan algebra $J(\N,\#,1_X)$ with unit $1_X$ and product
\begin{equation}
\label{jprod}
x\cdot y = \frac{1}{2} ( x \# y + \Tr(x)y + \Tr(y)x - \Sr(x,y)1_X ). 
\end{equation}
	\end{prop}

Let $X$ be a degree $3$ associative algebra defined over $k$ that is equipped with a cubic norm form $\n$, trace form $\tr$, and $\varsigma \in G_m(k)$.  We define a $k$-module $J = X_0 \oplus X_1 \oplus X_2$ to be three copies of our associative algebra $X$ and make the following identifications for $a = (a_0,a_1,a_2), b=(b_0,b_1,b_2) \in J$,

\begin{align}
1_J &= (1_X, 0, 0) \\
\N\left( a \right) &= \n(a_0) + \varsigma \n(a_1) + \varsigma^{-1} \n(a_2) - \tr(a_0a_1a_2) \\
\Tr\left( a \right) &= \tr(a_0) \\
\Tr\left( a, b \right) &= \tr(a_0,b_0) + \tr(a_1,b_2) + \tr(a_2,b_1) \\
a^{\#} &= \left( a_0^{\#} -  a_1a_2, \varsigma^{-1} a_2^{\#} - a_0a_1, \varsigma a_1^{\#} - a_2a_0 \right).
\end{align}

This algebra with the product from (\ref{jprod}) gives rise to a unital Jordan algebra $J(X, \mu)$, and is called the \emph{first Tits construction}.  The algebra $J(\Mat_3(k), 1)$ is a split Albert algebra over $k$, and over a given field $k$ this is a representative of the only isomorphism class of split Albert algebras.

Two Jordan algebras $J$ and $J'$ are \emph{isotopic} if $J' \cong J^{\langle u \rangle}$, where $J^{\langle u \rangle}$ is the Jordan algebra with the same elements as $J$ and with the product
\[ x\langle u \rangle y = (x\cdot u)\cdot y + (y\cdot u)\cdot x - (x\cdot y)\cdot u = \{x,u,y\}. \]  
The $u$-isotope $J^{\langle u \rangle}$ of $J$ has the same underlying vector space as $J$ with a shifted multiplication thus shifting the identity and inverses by,
\begin{align}
e^{\langle u \rangle} &= u^{-1} \\
x^{\langle u, -1 \rangle} &= \U_u^{-1} x^{-1}.
\end{align}
The following fact taken directly from \citep{ga01} will be useful in narrowing down possible isomorphism classes of $k$-involutions.

\begin{lem}
\label{classalb}
If $J$ is a reduced Albert algebra, then every element $\lambda \in G_m(k)$ has a corresponding norm similarity $\phi_{\lambda}:J \to J$ such that $\N(\phi_{\lambda}(j)) = \lambda \N(j)$. 
\end{lem}
\begin{proof}
If we let
\[ h(\xi_1,\xi_2,\xi_3;a, b, c) = \begin{bmatrix}
\xi_1 & c & \gamma_1^{-1}\gamma_2 \overline{b} \\
\gamma_2^{-1}\gamma_1 \overline{c} & \xi_2 & a \\
b & \gamma_3^{-1}\gamma_2\overline{a} & \xi_3
\end{bmatrix}, \]
then $\phi_{\lambda}\big( h(\xi_1,\xi_2,\xi_3;a, b, c) \big) = h(\lambda \xi_1, \lambda \xi_2, \lambda^{-1} \xi_3; a, b, \lambda c)$.
\end{proof}

To study groups of type $\E_6$ we define the $k$-algebra $\mathcal{B}(J \times J,k \times k,\zeta) = \mathcal{B}(J^{\times 2},k^{\times 2},\zeta)$, where $J$ is a reduced $\mathcal{J}$-algebra over $k$ and $\zeta \in G_m(k)$.  This algebra takes the form
\[ \mathcal{B}(J^{\times 2},k^{\times 2},\zeta) = \left\{ \begin{bmatrix}
\alpha & j \\
l & \beta  \\ 
\end{bmatrix} \ \bigg| \ \alpha,\beta \in k \text{ and } j,l \in J \right\}, \]
with the following product
\[\begin{bmatrix}
\alpha_1 & j_1 \\
l_1 & \beta_1  \\ 
\end{bmatrix}
\begin{bmatrix}
\alpha_2 & j_2 \\
l_2 & \beta_2  \\ 
\end{bmatrix} = 
\begin{bmatrix}
\alpha_1\alpha_2  + \zeta \Tr(j_1,l_2) & \alpha_1 j_2 + \beta_2 j_1 + \zeta (l_1 \# l_2) \\
\beta_1 l_2 + \alpha_2  l_1 + (j_1 \# j_2) & \beta_1\beta_2 + \zeta \Tr(j_2,l_1)  \\ 
\end{bmatrix}. \]
We can define an algebra involution on our new algebra $(B,-)$ by, 
\[ \overline{\begin{bmatrix}
\alpha & j \\
l & \beta  \\ 
\end{bmatrix}} =
\begin{bmatrix}
\beta & j \\
l & \alpha  \\ 
\end{bmatrix}. \]
When $J$ is a split Albert algebra over a field $k$, $\mathcal{B}\left(J^{\times 2}, k^{\times 2}, \zeta \right)$ is a split Brown algebra over $k$, and will be denote by $B^d$.  If $(B,-)$ is a $k$-algebra with involution we say $(B,-)$ is a \emph{Brown algebra} if $(B,-) \otimes_k k_{sep} \cong B^d \otimes_k k_{sep}$ for $k_{sep}$, a separable closure of $k$.  A Brown algebra is called \emph{reduced} if the corresponding Albert algebra is reduced.

As pointed out in \citep{ga01} Brown began studying these algebras in greater generality than we require, using parameters he calls $\mu, \nu, \omega_1,\omega_2, \delta_1,\delta_2$ in \citep{br63}.  We are only interested in the case $\mu = \nu = \omega_1 = 1$ and $\omega_2 = \delta_1 = \delta_2 = \zeta=1$.  Every Brown algebra has a one dimensional space of elements that are skew symmetric with respect to $-:B \to B$.  If this space is $ks_0$ and $s_0^2 \in G_m(k)$ is a square in $k$ we say that $(B,-)$ is of \emph{type $1$}, and $(B,-)$ is of \emph{type $2$} otherwise.  Garibaldi goes on to show us the following result in \citep{ga01}, where $\Aut^+(B,-)$ is the identity component of $\Aut(B)$.  

\begin{thm}
If $(B, -)$ is a Brown algebra of type $m$ over $k$, then $\Aut^+(B,-)$ is a simply connected algebraic group of type $^m\hspace{-.07cm}E_6$ over $k$ with trivial Tits algebra.  Every simply connected group of type $E_6$ with trivial Tits algebra is isomorphic to a group of this form. 
\end{thm}

We do not define Tits algebras here, since all groups we will consider have trivial Tits algebras.  For a thorough treatment see \citep{kmrt}.  

\begin{lem}
\label{browniso1}
Any Brown algebra of type $1$ is isomorphic to an algebra of the form $\mathcal{B}(J^{\times 2},k^{\times 2},\zeta)$.
\end{lem}

Throughout the current paper we consider only split groups of type $\E_6$, and all split groups are of type $1$.  When we think of an element of $B$ as being of the form above, then we can take a typical element of the skew symmetric space to be of the following form 
\[ \begin{bmatrix}
		\alpha & \cdot \\
		\cdot & -\alpha \\
	\end{bmatrix}
= \alpha\begin{bmatrix}
		1 & \cdot \\
		\cdot & -1 \\
	\end{bmatrix}
= \alpha s_0. \]

We define $\Inv(J)$ as the set of linear bijections that leave the cubic norm form on $J$ invariant, 
	\[ \Inv(J) = \{ \phi \in \Lin(J) \ | \ \N(\phi(j)) = \N(j) \text{ for all } j \in J \}, \]
wher $\Lin(J)$ is the set of linear maps from $J$ to itself.  An element $\phi \in \Inv(J)$ acts on $B$ in the following way,
\[ \hat{\phi} \begin{bmatrix}
			\alpha & j \\
			l & \beta
			\end{bmatrix} =
			\begin{bmatrix}
			\alpha & \phi(j) \\
			\phi^{\dagger}(l) & \beta
			\end{bmatrix}, \]
where $\dagger: \Inv(J) \to \Inv(J)$ is the unique outer automorphism such that 
\[ \phi(j)\#\phi(l) = \phi^{\dagger}(j \#l). \]
Notice that when $\phi \in \Aut(J) \subset \Inv(J)$, $\phi(j)\#\phi(l) = \phi(j \#l)$.  Now we see through a series of known results from \citep{ga01} that our Brown algebra's structure is determined greatly by its corresponding Albert algebra. 
\begin{lem}
\label{browniso2}
$\mathcal{B}(J_1^{\times 2},k^{\times 2},\zeta_1) \cong \mathcal{B}(J_2^{\times 2},k^{\times 2},\zeta_2)$ if and only if there is a norm similarty of the form $\phi_{\zeta_1/\zeta_2}$ or $\phi_{\zeta_1/\zeta_2^2}$.
\end{lem}
In other words there exists a norm similarity on $J$ that has either $\frac{\zeta_1}{\zeta_2}$ or $\frac{\zeta_1}{\zeta_2^2}$ as its multiplier.  The two cases correspond to $\phi$ being induced by an element of $\Inv(J)$ or not, respectively.
\begin{lem}
\label{browniso3}
If $J$ is reduced then $\mathcal{B}(J,k,\zeta) \cong \mathcal{B}(J,k,1)$.
\end{lem}

All Albert algebras we will consider are reduced and so we can always assume, when dealing with a Brown algebra over a split Albert algebra that $\zeta=1$.

\begin{thm}
\label{brownsplit}
$\Aut^+(B^d)$ is the split simply connected group of type $\E_6$.
\end{thm}

We can construct $k$-involutions of $\Aut^+(B,-)$ from those of $\Aut(J)$ and $\Aut(C)$ where $B$, $J$ and $C$ are split Brown, Albert, and eight dimensional composition algebras respectively.  We will need the following result.  This is well known, and can be found in \citep{hu14}.

\begin{prop}
\label{g2f4}
Let $C$ be a composition algebra and $J$ a reduced Albert algebra over $C$, then $\Aut(C) \subset \Aut(J)$ is a subgroup.
\end{prop}
\begin{proof}
Since $J$ is reduced $J \cong \Her_3(C,\gamma)$.  We can see this by letting $t \in \Aut(C)$ extend to the element $\hat{t} \in \Aut(J)$ such that
\[\hat{t}\left( h(\xi_1,\xi_2,\xi_3;a, b, c) \right) = h\left(\xi_1,\xi_2,\xi_3;t(a),t(b), t(c)\right). \]
\qed\end{proof}

\begin{prop}
\label{f4e6}
Let $J$ be a reduced Albert algebra over $k$ with composition algebra $C$ and $B=\mathcal{B}(J^{\times 2},k^{\times 2},\zeta)$, then $\Aut(J)$ is a subgroup of $\Aut^+(B,-)$.
\end{prop}
\begin{proof}
Let $\phi \in \Aut(J)$ and extend $\phi$ to $\hat{\phi} \in \Aut^+(B,-)$ by
\[ \hat{\phi}\left( \begin{bmatrix}
					\alpha & j \\
					l & \beta \\
			            \end{bmatrix} \right) = 
				  \begin{bmatrix}
					\alpha &\phi(j) \\
					\phi(l) & \beta \\
			            \end{bmatrix}. \]
And notice that if $\phi \in \Aut(J)$ then $\phi(j) \# \phi(j') = \phi(j\#j')$ and $\Tr(\phi(j),\phi(j'))=\Tr(j,j')$.  Also, this map leaves the algebra involution on $B$ invariant.
\qed\end{proof}

Springer and Veldkamp \citep[7.3.2]{sv00} refer to the $26$ dimensional irreducible algebraic variety consisting of norm $1$ Albert algebra elements as
	\[ W(K) = \{ x \in J \otimes_k K \ | \ \N(x) = 1 \}. \]
 It is known $\Inv(J)$ acts transitively on $W(K)$, and $\Inv(J)$ is a quasisimple group of type $\E_6$ having dimension $26+52=78$. The algebraic group $\Inv(J)$ is defined over $k$.  Notice the dimension of $\Aut(J) \subset \Inv(J)$ is $52$.

\section{Brown algebras and subalgebras}

We continue with the technique from \citep{hu12,hu14} of classifying $k$-involutions of an automorphism group by classifying the subalgebras up to isomorphism that are fixed elementwise by an element of order $2$ in our automorphism group.  To this end we discuss the structure of some of these subalgebras.  By a \emph{subalgebra}, or a \emph{Brown subalgebra}, we mean a subset that is closed with respect to the product and closed under the restriction of the algebra involution.

\begin{prop}
The algebra of the following form 
\[ \left\{ \begin{bmatrix}
	\alpha & \phi_1(j) \\
	\phi_2(j) & \alpha 
	\end{bmatrix} \ \bigg| \ \alpha \in k, j \in J, \phi \in \Aut(J) \right\}, \]
is a Brown subalgebra of a reduced Brown algebra over $k$ if and only if $\phi_1$ and $\phi_2$ have order $2$ and  $\phi_1\phi_2=\phi_2\phi_1$.
\end{prop}
\begin{proof}
We take the product of two elements in the above set
\begin{align*}
&\begin{bmatrix}
	\alpha & \phi_1(j) \\
	\phi_2(j) & \alpha 
	\end{bmatrix}
	\begin{bmatrix}
	\beta & \phi_1(l) \\
	\phi_2(l) & \beta 
	\end{bmatrix} \\
	&\ \ \ \ =\begin{bmatrix}
	\alpha\beta + \Tr(\phi_1(j),\phi_2(l) ) & \alpha \phi_1(l) + \beta\phi_1(j) + \phi_2(j)\# \phi_2(l) \\
	\alpha \phi_2(l) + \beta\phi_2(j) + \phi_1(j)\# \phi_1(l) & \alpha\beta + \Tr(\phi_1(l),\phi_2(j) ) 
	\end{bmatrix} \\
	\\
	&\ \ \ \ =\begin{bmatrix}
	\alpha\beta + \Tr(\phi_1(j),\phi_2(l) ) & \alpha \phi_1(l) + \beta\phi_1(j) + \phi_2(j)\# \phi_2(l) \\
	\alpha \phi_2(l) + \beta\phi_2(j) + \phi_1(j)\# \phi_1(l) & \alpha\beta + \Tr(\phi_2\phi_1(l),j ) 
	\end{bmatrix} \\
	 \\
	&\ \ \ \ =\begin{bmatrix}
	\alpha\beta + \Tr(\phi_1(j),\phi_2(l) ) & \phi_1\Big(\alpha l + \beta j + \phi_1 \big(\phi_2(j)\# \phi_2(l) \big)\Big)\\
	\phi_2 \Big(\alpha l + \beta j + \phi_2 \big(\phi_1(j)\# \phi_1(l) \big)\Big)& \alpha\beta + \Tr(\phi_2(l),\phi_1(j) ) 
	\end{bmatrix} ,
	\end{align*}
since the bilinear form is symmetric, $\Tr(\phi(j), \phi(l)) = \Tr(j,l)$ when $\phi \in \Aut(J)$, the automorphisms commute, and the automorphisms are of order $2$
	\begin{align*}
	&\begin{bmatrix}
	\alpha & \phi_1(j) \\
	\phi_2(j) & \alpha 
	\end{bmatrix}
	\begin{bmatrix}
	\beta & \phi_1(l) \\
	\phi_2(l) & \beta 
	\end{bmatrix} \\
	&\ \ \ \ =\begin{bmatrix}
	\alpha\beta + \Tr(\phi_1(j),\phi_2(l) ) & \phi_1\Big(\alpha l + \beta j + \phi_1 \phi_2\big( j\# l \big)\Big)\\
	\phi_2 \Big(\alpha l + \beta j + \phi_1\phi_2 \big(j \# l \big)\Big)& \alpha\beta + \Tr(\phi_1(j),\phi_2(l) ) 
	\end{bmatrix}.
\end{align*}
\qed\end{proof}

The algebra $\mathcal{B}(J,k)$ takes the following form
\[ \mathcal{B}(J,k) = \left\{ \begin{bmatrix}
					\alpha & j \\
					j & \alpha \\
				   \end{bmatrix} \ \bigg| \ \alpha \in k, j \in J \right\} \]
when embedded in a reduced Brown algebra.  This algebra is fixed elementwise by
\[ \varpi \left(\begin{bmatrix}
					\alpha & j \\
					l & \beta \\
				   \end{bmatrix} \right) =
					\begin{bmatrix}
					\beta & l \\
					j & \alpha \\
				   \end{bmatrix}. \]
This map induces the outer automorphism in $\Inv(J)$ the quasisimple group of type $\E_6$.  

Let $\dagger: \Inv(J) \to \Inv(J)$ be the automorphism that takes $\phi \mapsto \phi^{\dagger}$ the unique linear map such that
\[ \Tr \big(\phi(x), \phi^{\dagger}(y) \big) = \Tr \big( x,y \big) \]
for all $x,y \in A$.  It is known that the trace form exhibits $\U$-symmetry,
\begin{equation}
\label{tracesym}
\Tr(\U_x y, z) = \Tr( y, \U_x z).
\end{equation}
Using this idea we can prove the following result.

\begin{lem}
\label{outerdef}
If $\U_x \in \Inv(J)$, then $\U_x^{\dagger} = \U_x^{-1}= \U_{x^{-1}}$.
\end{lem}
\begin{proof}
From the above discussion
\[ \Tr(\U_x y, \U_x^{\dagger} z) = \Tr(y,z), \]
for all $y,z \in J$.  Also, recall that $\Tr$ is nondegenerate and $\U_x:J \to J$ is a bijection since $\N(x) \neq 0$.  Using \ref{tracesym} we see that
\[ \Tr(\U_x y, \U_x^{\dagger} z) = \Tr( y, \U_x\U_x^{\dagger} z) = \Tr(y,z), \]
so $\U_x\U_x^{\dagger} z = z$ for all $z \in J$, and $\U_x^{\dagger} = \U_x^{-1} =\U_{x^{-1}}$.
\qed\end{proof}

Also worth noticing is that when $J$ is split and $\U_x \in \Inv(J),$
\[ \U_x^{\dagger} = \U_{x^{-1}} = \U_{x^{\#}}, \]
since $x^{-1} = \N(x)^{-1} x^{\#}$ and $\N(x) =1$.
This outer automorphism can also be characterized by the property
\[ \phi(x) \# \phi(y) = \phi^{\dagger}(x \# y) \]
for all $x,y \in A$.  For a deeper discussion and proofs see \citep[7.3]{sv00}.  It is clear that this is the same map from \ref{browniso2}, which is implied by the notation.  To see that $\varpi$ is in $\Aut^+(B,-)$ induces $\dagger$ on $\Inv(J)$ we simply notice
\[ \varpi\hat{\phi} \varpi \begin{bmatrix}
					\alpha & j \\
					l & \beta \\
				   \end{bmatrix} = 
\varpi \hat{\phi}\begin{bmatrix}
					\beta & l \\
					j & \alpha \\
				   \end{bmatrix} =
	\varpi \begin{bmatrix}
					\beta & \phi(l) \\
					\phi^{\dagger}(j) & \alpha \\
				   \end{bmatrix} =
				\begin{bmatrix}
					\alpha & \phi^{\dagger}(j) \\
					\phi(l) & \beta \\
				   \end{bmatrix}, \]
and so $\varpi \hat{\phi} \varpi = \widehat{\phi^{\dagger}}$. \\

When we consider $\Inv(J)$ with $J$ constructed using the first Tits construction we can realize the action of the outer automorphism on a maximal split torus as $\dagger:\Inv(J) \to \Inv(J)$.  Consider $\varphi :\SL_3 \times \SL_3 \times \SL_3 \hookrightarrow \Inv\left(J(\Mat_3( \  \ ),1)\right)$, then $\varphi$ takes the following form
\[ \varphi(u,v,w)(a_0,a_1,a_2) = \left(u a_0   v^{-1}, v  a_1  w^{-1}, w  a_2  u^{-1} \right).  \]
Through straight forward computations we can verify that $\varphi(u,v,w)$ leaves the norm form on $J(\Mat_3(k),1)$ invariant.  From this action we can construct a $k$-split maximal torus.  The outer automorphism group is of the from 
\[ \Out\big(\Inv(J)\big)=\{ \id, \dagger \}, \]
and the nontrivial representative offered in \citep{sp73} is of the form
\[ \varphi^{\dagger}(u,v,w) = \varphi(v,u,w). \]

\section{The structure of $\Aut^+(B,-)$}

In \citep{mcc69} we see that in general the identity
	\[ \N\left( \U_x y \right) \U_x y = \N(x)^2 \N(y) \U_x y, \]
is true.
When $k$ is an integral domain (or more generally $k$ has no $3$-torsion or $k$ has no nilpotent elements \citep{mcc05}) we have the identity
\begin{equation}
\N\left( \U_x y \right) = \N(x)^2 \N(y), 
\end{equation}
and this gives us a family of elements that leave $\N:J \to k$ invariant when $\N(x)^2 = 1$.  

\begin{prop}
\label{elinvj}
If $\phi \in \Aut(J)$, $\phi \U_x = \U_{\phi(x)} \phi$.
\end{prop}

\begin{proof}
Let $\phi\in \Aut(J)$ and $x \in J$.  Let $y \in J$
	\begin{align*}
	\phi\left( \U_x y \right) &= \phi \left( 2\Lm_x^2(y) - \Lm_{x^2}(y) \right) \\
		&=  2 \phi(x) \cdot \left(\phi(x)\cdot \phi(y) \right) - \phi(x)^2\cdot \phi(y) \\
		&= \U_{\phi(x)} \phi(y).
	\end{align*}
\qed\end{proof}

The group of all isotopies from $J$ to itself is a group called the \emph{structure group} denoted by $\Gamma(J)$.  There is a subgroup of $\Gamma(J)$ known as the \emph{inner structure group},
\[ \Gamma_1(J) = \langle \U_x |  \N(x) \neq 0 \rangle. \]
In \citep{ja61} Jacobson defines the norm preserving group of a Jordan algebra $J$, which we are denoting by $\Inv(J)$.  Jacobson also defines the following subgroups of $\Inv(J)$
\[ \Inv_1(J) = \{ \U_{a_r} \cdots \U_{a_2}\U_{a_1} \ | \ \N(a_r)^2\cdots \N(a_2)^2 \N(a_1)^2 = 1 \} \]
and 
\[ \Inv_2(J) = \{ \U_{a_r} \cdots \U_{a_2}\U_{a_1} \ | \ \N(a_r)\cdots \N(a_2) \N(a_1) = 1 \}. \]
Referring to $\Inv_2(J)$ as the reduced norm preserving group.  Notice that 
\[ \Inv(J) \supset \Inv_1(J) \supset \Inv_2(J). \]
Clearly $\Gamma_1(J) \supset \Inv_1(J)$, and in \citep[I-12]{ja68} it is left as an exercise to show the following.

\begin{prop}
\label{gamform}
Every element of $\Gamma_1(J)$ is of the form $\U_x^{\langle y \rangle}$ for $x$ and $y$ such that $\N(x)\N(y) \neq 0$.
\end{prop}
This is seen to be true by applying the following two properties of $\U$-operators related to  isotopes of $J$,
\begin{align}
\U_x^{\langle y \rangle} &= \U_x\U_y \label{isou} \\
\big(J^{\langle y \rangle} \big)^{\langle z \rangle} &= J^{\langle \U_y(z) \rangle}. \label{isoiso}
\end{align}
We can say something further in the case when $J$ is a split simple Albert algebra.  In particular the following result appears as \citep[Theorem 9]{ja61}.

\begin{thm}
\label{l2inv}
If $J$ is a split exceptional simple Jordan algebra the reduced norm preserving group coincides with the norm preserving group, i.e. 
\[ \Inv_2(J) = \Inv(J). \]
\end{thm}

\begin{prop}
\label{comp}
For an Albert algebra $J$, $J^{\langle y \rangle}$ is reduced if and only if $J$ is reduced.  Moreover, the octonion algebra associated with $J^{\langle y \rangle}$ is isomorphic to the octonion algebra associated with $J$.
\end{prop}

Since any two Albert algebras over split octonion algebras taken over the same field have the same norm preserving group, and it is also known that all split Albert algebras over a given field are isomorphic.

\begin{cor}
\label{topemorph}
If $J^{\langle y \rangle}$ is an isotope of a split Albert algebra $J$, then $\Inv(J^{\langle y \rangle}) \cong \Inv(J)$, moreover $J \cong J^{\langle y \rangle}$.
\end{cor}

Notice if we have an element $\phi \in \Inv(J)$ and $\phi^2 = \id$, then  
\[ \phi = \U_{a_r}\cdots \U_{a_2}\U_{a_1} \text{ and } \N(a_r)^2 \cdots \N(a_2)^2\N(a_1)^2 = 1. \]
Furthermore, if $J$ is split by \ref{l2inv} we can choose $a_1,a_2,\ldots,a_r \in J$ such that $ \N(a_r) \cdots\N(a_2)\N(a_1) = 1$.  By \ref{gamform} we see that 
\[ \U_{a_r}\cdots \U_{a_2}\U_{a_1} = \U_x^{\langle y \rangle}. \]
for some $x,y \in J$.

\begin{prop}
\label{E6inv}
When $J$ is split all elements of order $2$ in $\Inv(J)$ correspond to an automorphism of $J^{\langle y \rangle}$ for some $y\in J$ with $\N(y) \neq 0$.
\end{prop}
\begin{proof}
For $J$ a split Albert algebra let $\varphi \in \Inv(J)$, and $\varphi^2 = \id$.  Then by \ref{gamform} $\varphi = \U_x^{\langle y \rangle}$ for some $x,y \in J$ and $\N(x)\N(y)=1$ by \ref{l2inv}.  In shifting from $J$ to its isotope $J^{\langle y \rangle}$ our outer automorphism shifts from $\U_x^{\dagger} = \U_{x^{-1}} \mapsto \U_x^{\langle y, \dagger \rangle} = \U_{x^{\langle y, -1 \rangle}}^{\langle y \rangle}$.  Notice that
\begin{align*}
\U_x^{\langle y, \dagger \rangle} &= \U_{x^{\langle y, -1 \rangle}}^{\langle y \rangle} \\
&= \U_{\U_y^{-1}x^{-1}}^{\langle y \rangle} \\
&=\U_{\U_y^{-1}x^{-1}} \U_y \\
&=\U_{y^{-1}} \U_{x^{-1}} \U_{y^{-1}} \U_y \\
&=\U_x \U_y \\
&= \U_x^{\langle y \rangle},
\end{align*}
and so $\U_x^{\langle y \rangle} \in \Aut(J^{\langle y \rangle})$.
\qed\end{proof}

In the following section we consider how $k$-involutions of type $\E_6$ correspond to subalgebras of Brown algebras.

\section{Isomorphism classes of $k$-involutions}

Let us first consider $k$-involutions that are inner involutions of an automorphism group of a $k$-algebra $\mathcal{A}$.  In this case the isomorphism class of a $k$-involution corresponds to an isomorphism class of a subalgebra of $\mathcal{A}$.  A version of the following Lemmas \ref{order2fixlem1} and \ref{order2fixlem2} appear in \citep{hu14} as a single result with a weaker hypothesis.  The proof appearing there is incorrect for the stated result.  We are thankful to the referee for pointing this out.  The main results that appear in \citep{hu14} follow from \ref{order2fixlem1} and \ref{order2fixlem2}.

\begin{lem}
\label{order2fixlem1}
Let $\mathcal{A}$ be an algebra.  If $t, t' \in \Aut(\mathcal{A})$ are $\Aut(\mathcal{A})$-conjugate, then their respective fixed point subalgebras $\mathcal{D}$,$\mathcal{D}' \subset \mathcal{A}$ are isomorphic.
\end{lem}

\begin{proof}
We start by assuming there exists a $g\in \Aut(\mathcal{A})$ such that $gt=t'g$.   Take $a\in \mathcal{D}$, the fixed point subalgebra of $\mathcal{A}$ with respect to $t$,
\begin{align*}
	gt(a) &= t'g(a) \\
	g(a) &= t'g(a), \\
\end{align*}
and $g(a) \in \mathcal{D}'$, the fixed point subalgebra with respect to $t'$.  Now we just reverse the argument using the fact that $g\in \Aut(\mathcal{A})$ and $tg^{-1} = g^{-1} t'$.  This shows us that $g^{-1}(a') \in \mathcal{D}'$ for all $a'\in \mathcal{D'}$, and so $g(\mathcal{D}) = \mathcal{D}'$.
\qed\end{proof}

Before we look at the partial converse we consider the following situation.

\begin{prop}
\label{order2decomp}
Let $\mathcal{A}$ is a $k$-algebra where $k$ is a field and $\ch(k) \neq 2$.  If $t\in \Aut(\mathcal{A})$, $t^2=\id$, $D$ is the subalgebra of $\mathcal{A}$ fixed by $t$, and
\[ \mathcal{A} = \mathcal{D} \oplus \mathcal{D}^{\perp}, \]
with respect to a nondegenerate bilinear form that is left invariant by $t$, then for $b \in D^{\perp}$ we have $t(b) = -b$.
\end{prop}
\begin{proof}
Let $\mathcal{A} = \mathcal{D} \oplus \mathcal{D}^{\perp}$ such that $\mathcal{D}$ is the fixed point subalgebra of $t \in \Aut(\mathcal{A})$ such that $t^2 =\id$.  Let $b\in D^{\perp}$
\[ t(b + t(b) ) = t(b) + b = b+t(b) \Rightarrow b + t(b) \in D. \]
Now if we choose $a \in \mathcal{D}$ we see that
\[ \langle a, b+t(b) \rangle = \langle a,b \rangle + \langle t(a),t(b) \rangle = 0, \]
since $t$ is an isometry, so $b + t(b) \in \mathcal{D}^{\perp}$.  We have $b+t(b) \in \mathcal{D} \cap \mathcal{D}^{\perp} = \{ 0 \} \Rightarrow t(b) = -b$.
\qed\end{proof}

The decomposition in the above proposition is the same as having a $\mathbb{Z}_2$-grading on $\mathcal{A}$ induced by an automorphism.  The connection between $k$-involutions and $\mathbb{Z}_2$-gradings are discussed further in section $7$.  For the partial converse of \ref{order2fixlem2} we need a stronger hypothesis.  In the following Lemma we assume that $\mathcal{A}$ is an algebra over a field of characteristic not $2$.  A similar result may exist with a less confining hypothesis, but for us this is enough.

\begin{lem}
\label{order2fixlem2}
Let $\mathcal{A}$ be a $k$-algebra with a nonsingular bilinear form where $k$ is a field not of characteristic $2$.  Assume $t,t' \in \Aut(\mathcal{A})$ are elements of order $2$ that have fixed point subalgebras $\mathcal{D}, \mathcal{D}'$ respectively, and $g(\mathcal{D})=\mathcal{D}'$ for $g \in \Aut(\mathcal{A})$.  If $g$ leaves the bilinear form on $\mathcal{A}$ invariant and
\[\mathcal{A}=\mathcal{D} \oplus \mathcal{D}^{\perp} = \mathcal{D}' \oplus\mathcal{D}'^{\perp}. \]
Then $t$ and $t'$ are $\Aut(\mathcal{A})$-conjugate.
\end{lem}

\begin{proof}
Let $\mathcal{D},\mathcal{D}' \subset \mathcal{A}$ such that $\mathcal{D}$ and $\mathcal{D}'$ are the fixed point subalgebras of $t, t' \in \Aut(\mathcal{A})$ respectively, and $t^2 = t'^2 = \id$.  Let $g\in \Aut(\mathcal{A})$ be such that $g(\mathcal{D})=\mathcal{D}'$.  Suppose $a \in \mathcal{D}$ and $b\in \mathcal{D}^{\perp}$.  We assume that $\mathcal{A}$ has the following decomposition, $\mathcal{A}=\mathcal{D} \oplus \mathcal{D}^{\perp} = \mathcal{D}' \oplus\mathcal{D}'^{\perp}$.
Notice $g(\mathcal{D}'^{\perp}) = \mathcal{D}^{\perp}$, and we take $a \in \mathcal{D}$ and $b \in \mathcal{D}^{\perp}$, then by \ref{order2decomp} we have
\[ t'g(a+b) = g(a) - g(b) = g(a-b) = gt(a+b). \]
\qed\end{proof}

In \citep{hu12} we saw that automorphisms of order $2$ of $\Aut(C)$, a split group of type $\G_2$, were determined by the isomorphism class of a quaternion subalgebra of the split octonion algebra $C$.  We can construct a $k$-split maximal torus of $\Aut(C)$ as being of the form

\[ T_C = \left\{ t_{(\eta,\nu)}=\diag(1,\eta \nu, \eta^{-1} \nu^{-1}, 1, \nu^{-1}, \eta, \eta^{-1}, \nu ) \ \big| \ \eta, \nu \in k^* \right\}. \]

Next we recall results from split groups of type $\G_2$ and $\F_4$ as they extend to our group of type $\E_6$.  Recall that the automorphism group of $\Aut(C)$ and $\Aut(J)$, where $C$ is a composition algebra and $J$ is an Albert algebra, consists of inner automorphisms only.

Then there exists $\theta$ and a maximal $k$-split torus in $\Aut(C)$ such that $T_C = \left(T_C \right)^-_{\theta}$, details of which are found in \citep{hu12}.  Specifically $\theta = \I_{t_*}$ where
\[ t_*= \begin{bmatrix}
\cdot & \cdot & \cdot & 1 \\
\cdot & \cdot & 1 & \cdot\\
\cdot& 1 & \cdot & \cdot \\
1 & \cdot & \cdot & \cdot
\end{bmatrix}
\bigoplus
\begin{bmatrix}
\cdot & \cdot & \cdot & 1 \\
\cdot & \cdot & 1 & \cdot \\
\cdot & 1 & \cdot & \cdot \\
1 & \cdot & \cdot & \cdot
\end{bmatrix}.
\]
The $k$-inner elements can all be found in $T$, and $t_*t$ fixes a quaternion algebra over the given field.  The number of these isomorphism classes and the elements of $T_C$ which are induced by automorphisms of $C$ that fix these quaternion subalgebras depend on $k$.

\begin{thm}
\label{G2res}
The classification of $k$-involutions of $\Aut(C)$ for some specific fields is as follows.
\begin{enumerate}[$(1)$]
\item When $k=K$ and $\mathbb{F}_p$ with $p>2$, there is only one isomorphism class of $k$-involutions of $\Aut(C)$ with $k$-involution $\theta$.
\item When $k=\mathbb{R}$ and $\mathbb{Q}_2$,  $\theta$ and $\theta \mathcal{I}_{ t{(1,-1)} }$ are representatives of the two isomorphism classes of $k$-involutions of $G=\Aut(C)$.
\item When $k=\mathbb{Q}_p$ with $p> 2$, $\theta$ and $\theta \mathcal{I}_{t(-Z_p,-pZ_p^{-1})}$ give us the two isomorphism classes.
\item When $k=\mathbb{Q}$, $\theta$ fixes a split quaternion algebra.  However, there are an infinite number of isomorphism classes of quaternion division subalgebras, and so an infinite number of classes of $k$-involutions, (one infinite collection is given by $D \cong \left(\frac{-1,p}{\mathbb{Q}} \right)$ and fixed by $\theta \I_{t{(p,1)}}$ when $p$ is a distinct prime such that $p \equiv 3 \mod 4$).
\end{enumerate}
\end{thm}

We think of a split group of type $\F_4$ as an automorphism group of an Albert algebra $J$.  When $J$ is split we can make the identification of the Albert algebra with the $3 \times 3$, $\gamma$-hermitian matrices over $C$, $\Her_3(C,\gamma)$, where $C$ a split octonion algebra.  We can actually do this any time the algebra is reduced, but here we only consider the split case.  There are two main types of subalgebras of $J$ fixed by an element of order $2$ in $\Aut(J)$.  Type (I) fixes a subalgebra isomorphic to $\Her_3(D,\gamma')$ where $D$ is a quaternion subalgebra of $C$, and type (II) fixes an $11$ dimensional quadratic subalgebra.  We follow \citep{ja68} and refer to them as type (I) and type (II) respectively.

Up to $\gamma$ the $k$-involutions of type (I) are isomorphic to $\hat{t}\in \Aut(J)$, where $t \in \Aut(C)$ and
\[ \hat{t}(\xi_1,\xi_2,\xi_3;a,b,c) = (\xi_1,\xi_2,\xi_3;t(a),t(b),t(c)), \]
as pointed out in \ref{g2f4}.  We can construct a $k$-split maximal torus in a way similar to our type $\E_6$ construction.  We take the action of $\varphi(u,u,v)$ on the first Tits construction of the split Albert algebra, and all of the isomorphism classes of the $k$-inner elements are determined by the isomorphism classes of $D$ and $\Her_3(D,\gamma')$.  

\begin{thm}
\label{f4expl}
For the following fields we can take as representatives of isomorphism classes of $k$-involutions of $\Aut(J)$ to be of the form $\theta \I_{\hat{t}}$ where $\theta$ is as in \ref{G2res}, and $t=t(u_1,u_2,v_1,v_2)$
\begin{enumerate}[$1.$]
\item $k=K$ or $k=\mathbb{F}_p$ where $p>2$, $t=t(1,1,1,1)$ is a representative of the only isomorphism class,
\item $k=\mathbb{R}$ for $D$ split we can choose $t(1,1,-1,1)$, for the positive definite case we can choose $t(1,1,1,1)$, and for the indeterminate quadratic form we can choose $t(-1,1,1,1)$,
\item $k=\mathbb{Q}_2$ we can choose $t(1,1,-1,1)$ for the split case and $t(1,1,1,1)$ for $D$ a division algebra,
\item $k=\mathbb{Q}_p$ with $p>2$ we can choose $t(1,1,-1,1)$ for $D$ split, and $t(1,1,-p,-Z_p)$ for $D$ a division algebra.
\end{enumerate}
\end{thm}

\begin{cor}
\label{F4res}
Let $C$ be a split octonion algebra with quadratic form $\q$ and $D$ a quaternion subalgebra with quadratic form $\q_D$.  Then regarding $k$-involutions of type (I),
\begin{enumerate}[$1.$]
\item if $k=K$ is algebraically closed there is one isomorphism class of the form $\Her_3(D,\gamma)$,
\item if $k=\mathbb{F}_p$ where $p$ is not even there is one isomorphism class of algebras of the form $\Her_3(D,\gamma)$,
\item if $k=\mathbb{R}$ there are $3$ isomorphism classes of algebras; one of the form $\Her_3(D,\gamma)$ corresponding to $D$ being split,  one when $D$ is a division algebra with $\gamma=\id$, and one when $D$ is a division algebra with $\gamma=(-1,1,1)$,
\item if $k=\mathbb{Q}_p$ there are $2$ isomorphism classes of algebras of the form $\Her_3(D,\gamma)$ corresponding to $D$ being split or $D$ being a division algebra.
\item if $k=\mathbb{Q}$ there are an infinite number of isomorphism classes.
\end{enumerate}
Regarding $k$-involutions of type (II) for a split Albert algebra, there is only one isomorphism class for a given field \citep{ja68}. 
\end{cor}

The class of $k$-involutions of type (II) have as a representative $\U_{s_{23}} = \I_{s_{23}}$, where
\[ s_{23} = \begin{bmatrix}
			1 & \cdot & \cdot \\
			\cdot & \cdot & 1 \\
			\cdot & 1 & \cdot \\
			\end{bmatrix}, \]
when considering the Albert algebra as having the form $\Her_3(C,\id)$.  There is a deeper discussion and proof of this result included in \citep{hu14}.  Now we can notice that
\[  \theta \varphi^{\dagger}(u,v,w) = \varphi(u^{-1},v^{-1},w^{-1}) = \left(\varphi(u,v,w)\right)^{-1}, \]
where $\theta$ is of the same form as in \ref{G2res}.  We have shown the following result.

\begin{prop}
\label{e6kinvtorus}
There is a $k$-split maximal torus $T \subset \Aut^+(B,-)$ such that $T$ is a maximal $(\theta \mathcal{I}_{\varpi},k)$-split torus, i.e. $T_{\theta\mathcal{I}_{\varpi}}^-=T$.
\end{prop}

Let us consider the subalgebras of $B$ fixed by elements of order $2$ in $\Aut^+(B,-)$ up to isomorphism.  All elements of order $2$ in the automorphism group of $\Aut^+(B,-)$ will be one of the following forms $\sigma = \mathcal{I}_{\hat{s}}$, $\theta \mathcal{I}_{\hat{t}}$, $\dagger = \mathcal{I}_{\varpi}$, $\sigma \dagger$ and $\theta \dagger$ where $\dagger$ is the outer automorphism and $s$, $t_{\gamma} \in \Aut(J)$.  The map $\hat{t}$ corresponds to $t_{\gamma} \in \Aut(J)$ of type (I), and $s \in \Aut(J)$ is of type (II).  Note that up to $\gamma$, $t \in \Aut(C)$ fixes a quaternion subalgebra $D \subset C$.

\begin{lem}
\label{fixJ}
The subalgebras of $B$ fixed by elements of order $2$ in $\Aut^+(B,-)$ take one of the following forms.
\begin{align}
B^{\hat{s}} &= \left\{ \begin{bmatrix}
							\alpha & j' \\
							l' & \beta
						\end{bmatrix} \ \bigg| \ \alpha, \beta \in k \text{ and } j',l' \in \beth \right\} \\
B^{\hat{t}} &= \left\{ \begin{bmatrix}
							\alpha & j' \\
							l' & \beta
						\end{bmatrix} \ \bigg| \ \alpha, \beta \in k \text{ and } j',l' \in \Her_3(D,\gamma) \right\} \\
B^{\varpi} &= \left\{ \begin{bmatrix}
							\alpha & j \\
							j & \alpha
						\end{bmatrix} \ \bigg| \ \alpha \in k \text{ and } j \in J \right\} \\ &\cong B^{\hat{s}\varpi} = \left\{ \begin{bmatrix}
							\alpha & j \\
							s(j) & \alpha
						\end{bmatrix} \ \bigg| \ \alpha \in k \text{ and } j \in J  \right\} \\
B^{\hat{t}\varpi} &= \left\{ \begin{bmatrix}
							\alpha & j \\
							t(j) & \alpha
						\end{bmatrix} \ \bigg| \ \alpha \in k \text{ and } j \in J \right\}
\end{align}
\end{lem}

This gives us four isomorphism classes of $k$-involutions up to $\q_D$ where $D$ is the quaternion subalgebra fixed by $\hat{t}|_C = t$, and the classes of $\q_D$ have been determined and are included in \citep{ ja58, sv00, hu14}.  Note, regardless of $k$, $\gamma$ no longer determines isomorphism classes of $k$-involutions of $\Aut^+(B,-)$ for elements in $k^*$ not represented by $\q_D$ as is the case for $\Aut(J)$, since $\mathcal{I}_{\iota}$ is an isomorphism of norm preserving groups when $\iota=\mathcal{I}_{\gamma}$ as noted in \citep{ja61}.

\begin{thm}
\label{E6kinv}
Let $C$ be a split octonion algebra and $D$ a quaternion subalgebra fixed by $t \in \Aut(C)$.  Let $\sigma = \mathcal{I}_{\hat{s}}$ and $\dagger = \mathcal{I}_{\varpi}$.  The following are representatives of $\Aut^+(B,-)$-conjugacy classes
\begin{enumerate}[$1.$]
\item if $k=K$ is algebraically closed there are four classes of $k$-involutions, one for each of the representatives $\sigma$, $\theta$, $\dagger$ and $\theta \dagger$,
\item if $k=\mathbb{F}_p$ (where $p$ is not even) there are four classes of $k$-involutions, one for each $\sigma$, $\theta$, $\dagger$ and $\theta \dagger$,
\item if $k=\mathbb{R}$ there are six isomorphism classes of $k$-involutions, one each for representatives $\sigma$ and $\dagger$; there are two isomorphism classes each for $\theta \mathcal{I}_{\hat{t}}$ and $\theta \mathcal{I}_{\hat{t}} \dagger$ respectively depending on whether $D$ is split or division,
\item if $k=\mathbb{Q}_p$ there are six isomorphism classes of $k$-involutions, one each for representatives $\sigma$ and $\dagger$; there are two isomorphism classes each for $\theta \mathcal{I}_{\hat{t}}$ and $\theta \mathcal{I}_{\hat{t}} \dagger$ respectively depending on whether $D$ is split or division,
\item if $k=\mathbb{Q}$ there is one class of $k$-involutions each with representative $\sigma$ and $\dagger$; there are an infinite number of classes for $\theta \mathcal{I}_{\hat{t}}$ and $\theta \mathcal{I}_{\hat{t}} \dagger$ respectively corresponding to one split class of quaternion subalgebras of $C$ and an infinite number of quaternion division subalgebras of $C$.
\end{enumerate}
\end{thm}

\begin{proof}
Notice isomorphism classes of Brown subalgebras do not depend on $\gamma$.  They also do not depend on $\zeta$, since we are considering only split Brown algebras.  So we only need to worry about cases where the class of quaternion algebra determines a new subalgebra fixed by an automorphism.  This only arises in the case where the invariant of type (1) is isomorphic to $\theta \mathcal{I}_{\hat{t}}$ or $\theta \mathcal{I}_{\hat{t}} \dagger$.  In the other cases we maintain the whole Albert algebra in some form, or the Albert subalgebra is of quadratic type and there is only one isomorphism class.  
\qed\end{proof}

It is straight forward to construct explicit representatives from the respective $\phi$-split and $(\phi,k)$-split tori using \ref{f4expl} and the discussion of maximal $k$-split tori of type $\E_6$ from section $3$.  We collect this information for some specific fields here in the following theorem.  Since there is only ever one isomorphism class that includes $\sigma$ or $\dagger$, we can use their description above, and focus on when the $k$-involution of type $\theta \mathcal{I}_{\hat{t}}$ or $\theta \mathcal{I}_{\hat{t}}\dagger$.

\begin{thm}
\label{E6kinv2}
The isomorphism classes of type $\theta$ and $\theta\dagger$ each have representatives of the form $\theta \mathcal{I}_{\hat{t}}$ and $\theta \mathcal{I}_{\hat{t}} \mathcal{I}_{\varpi}$, and depend only on 
\[ t=t(u_1,u_2,v_1,v_2,w_1,w_2) \in T_{\mathcal{I}_{\varpi}}^-. \]
Each of the following values of $t$ correspond to exactly one conjugacy class of type $\theta$ and exactly one conjugacy class of type $\theta \dagger$.
\begin{enumerate}[$1.$]
\item $k=\bar{k}$ there is only one conjugacy class, so we can use $t(1,1,1,1,1,1)$
\item $k=\mathbb{F}_p$ there is only one conjugacy class, so we can use $t(1,1,1,1,1,1)$
\item $k=\mathbb{R}$ or $\mathbb{Q}_2$ we can choose $t(1,1,1,1,-1,1)$ for the case when $D$ is split and $t(1,1,1,1,1,1)$ for $D$ a division algebra,
\item $k=\mathbb{Q}_p$ when $p>2$ we can choose $t(1,1,1,1,-1,1)$ for $D$ split, and $t(1,1,1,1,-p,-Z_p)$ for $D$ a division algebra.
\end{enumerate}
\end{thm}
\begin{proof}
This follows from the above discussion, \ref{e6kinvtorus}, and \ref{f4expl}.
\qed\end{proof}

\section{Fixed point groups}

In this section we will make repeated use of the following result, which can be found in \citep{hu12, hu14}.

\begin{lem}
\label{fixedlemma}
Let $t\in \Aut(\mathcal{A})=G$ such that $t^2 = \id$ and $\mathcal{D} \subset \mathcal{A}$ the subalgebra elementwise fixed by $t$ then $f \in  G^{\I_t} = \{g \in G \ | \ \I_t(g)=g \} \subset \Aut(G)$ if and only if $f$ leaves $\mathcal{D}$ invariant.
\end{lem}

We will denote $G = \Inv(J)$ and $G^{\mathcal{I}_f} = G^f$, where $f \in \Inv(J)$.  From this we can identify fixed point groups corresponding to each $k$-involution.  We let $\phi \in \Inv(J)$ such that $\phi^2 = \id$, then $\phi \in \Aut(J)$ and $i$ leaves some subalgebra $\mathcal{S} \subset J$ fixed.  By \ref{fixedlemma} $\theta \in G^\phi$ if and only if $\theta$ leaves the corresponding subalgebra invariant, and the same goes for $G^{\phi\varpi}$.

\begin{prop}
\label{outerfix}
An element $\delta \in \Inv(J)$ is in $G^{\phi\varpi}$ where $\phi^2 = \id$ if and only if $\phi \delta \phi = \delta^{\dagger}$.
\end{prop}

\begin{proof}
Let $\delta \in \Inv(J)$ then $\hat{\delta} \in \Aut^+(B,-)$ and
\[ \hat{\delta} \left( \begin{bmatrix}
		\alpha & j \\
		\phi(j) & \alpha
		\end{bmatrix}\right) =
		\begin{bmatrix}
		\alpha & \delta(j) \\
		\delta^{\dagger}\phi(j) & \alpha
		\end{bmatrix}. \]
We now have
\[ \begin{bmatrix}
		\alpha &\delta(j) \\
		\delta^{\dagger}\phi(j) & \alpha
		\end{bmatrix} \in \Aut^+(B,-)^{\phi\varpi } \]
if and only if $\delta^{\dagger}\phi (j) = \phi \delta(j)$.  This tells that 
\begin{equation}
\label{outfix}
\phi \delta \phi = \delta^{\dagger}, 
\end{equation}
and we arrive at equivalent statements when we consider $\varpi \phi$ or $\phi \varpi$.
\qed\end{proof}

\begin{cor}
\label{corouterfix}
The following three statements characterize the fixed point groups corresponding to outer $k$-involutions of $G=\Inv(J)$ for $\delta \in G$.
\begin{enumerate}[$1.$]
\item $\delta \in G^{\varpi}$ if and only if $\delta  = \delta^{\dagger}$ 
\item $\delta \in G^{\hat{s}\varpi}$ if and only if $ \delta = s\delta^{\dagger}s$ 
\item $\delta \in G^{\hat{t}\varpi}$ if and only if $ \delta  = t\delta^{\dagger}t$ 
\end{enumerate}
\end{cor}

Notice that $\varpi \in \Aut^+(B,-)$ is in the fixed point group of every outer $k$-involution.  Also, when $\delta \in \Aut(J)$ recall $\delta^{\dagger} = \delta$.  So equation \ref{outfix} reduces to 
\[ \phi \delta \phi = \delta, \]
which puts $\delta \in \Aut(J)^{\phi}$.  In case $3$ above we have 
\[ \delta = t \delta^{\dagger} t. \]
Recall that in $\Inv(J)$, $t_{\gamma} \cong t_{\gamma'}$ for all $\gamma, \gamma' \in G_m(k)^{\times 3}$.  So our $k$-involutions of type $t_{\gamma}$ will be written as $t_{\q}$, where $\q$ is the norm of a quaternion algebra.

\begin{prop}
The subalgebras $B^{\varpi}$, $B^{s \varpi}$ are isomorphic.
\end{prop}

\begin{proof}
We can take $\U_{S'} = s$ as the representative of the class of $k$-involutions fixing an $11$ dimensional quadratic subalgebra of $J$ where
\[ S' = \begin{bmatrix}
	1 & \cdot & \cdot \\
	\cdot & -1 & \cdot \\
	\cdot & \cdot & -1
	\end{bmatrix}. \]
Let us consider
\[ V =\begin{bmatrix}
	1 & \cdot & \cdot \\
	\cdot & \cdot & v \\
	\cdot & -v & \cdot
	\end{bmatrix}, \]
where $v \in C$ and $v^2=1$ such that $\overline{v}=-v$.  We know such a $v$ exists since $C$ is split.  Then $\U_V \in \Inv(J)$ and $\widehat{\U_V} : B \to B$ is an isomorphism of subalgebras $B^{\varpi}$ and $B^{ s \varpi}$.  First notice that $\U_V^{\dagger}=\U_{V^{-1}}=\U_V^{-1}$, and $\U_V^2 = s$.  And then $\U_{S'} \U_V = \U_V^{\dagger}$. 
So we have
\[ \widehat{\U_V} \begin{bmatrix}
					\alpha & j \\
					j & \alpha \\
				\end{bmatrix} =
				\begin{bmatrix}
					\alpha & \U_V(j) \\
					\U_V^{\dagger}(j) & \alpha \\
				\end{bmatrix} =
				\begin{bmatrix}
					\alpha & \U_V(j) \\
					s\U_V(j) & \alpha \\
				\end{bmatrix}. \]
\qed\end{proof}

\begin{lem}
\label{k^4}
If $\varphi \in \Inv(J)^s$, $\varphi(u) = g^4 u$ for some $g \in G_m(k)$.
\end{lem}

\begin{proof}
Let us consider $\varphi \in \Inv(J)^{s}$ not isomorphic to a non-trivial element of $\Spin(10,k)$, where $s$ fixes elementwise $\beth$ an $11$ dimensional quadratic subalgebra of $J$ corresponding to 
\[   u = \begin{bmatrix}
1 & \cdot &  \cdot \\
  \cdot &  \cdot &  \cdot \\
 \cdot &  \cdot &  \cdot
\end{bmatrix}. \]
Let $\varphi \in \Inv(J)^{\hat{s}}$ then the space $ku$ is left invariant by $\varphi$ and so there exists $h \in G_m(k)$ such that
\[ \varphi(u) = hu. \]
So $\varphi$ takes the form 
\[ \varphi(j) = \varphi \begin{bmatrix}
\xi_1 & c & \overline{b} \\
 \overline{c} & \xi_2 & a \\
b & \overline{a} & \xi_3
\end{bmatrix}=
\begin{bmatrix}
h\xi_1 & z & \overline{y} \\
 \overline{z} & \eta_2 & x \\
y & \overline{x} & \eta_3
\end{bmatrix}. \]
Our map $\varphi \in \Inv(J)$ and so $\N(\varphi(j)) = \N(j)$.  Recall
\[ \N(j) = \xi_1\xi_2\xi_3 - \xi_1\q(a) - \xi_2\q(b) - \xi_3\q(c) + \q(ab, \overline{c}),  \]
and so
\[ \N(\varphi(j)) = h\xi_1\eta_2\eta_3 - h\xi_1\q(x) - \eta_2\q(y) - \eta_3\q(z) + \q(xy, \overline{z}). \]
Considering the space where $\xi_1 =1$ and $\xi_2=\xi_3=b=c=0$ we see that
\[ \q(a) = h\q(x), \]
and by the symmetry of the algebra $x=h_1^{-1} a$ or $x = h_1^{-1} \overline{a}$, either way there exists $h_1 \in G_m(k)$ such that $h_1^2 = h$.  Next we consider $\varphi$ acting on the space $\xi_1 = 1, a=b=c=0$, and $\xi_2=\pm\xi_3$.  Then $h\eta_2^2 = \xi_2^2$, and so $\eta_2 = h_1^{-1}\xi_2$ or $h_1^{-1}\xi_3$ and $\eta_3 = h_1^{-1}\xi_3$ or $h_1^{-1}\xi_2$ where $h_1$ is as before.  Our map $\varphi$ also fixes $\beth^{\perp} \subset J$, and so we consider the space $\xi_1 = 0, \xi_2=\xi_3=1, a=0$, and $b=c$.  So in order to leave $\N$ invariant we have
\[ h_1^{-1}(\q(y) + \q(y)) = \q(b) + \q(b) \Rightarrow h_1^{-1}\q(y) = \q(b), \]
and so as before we have $\q(y) = h_1\q(b)$.   By symmetry $y = h_2 b$ or $h_2\overline{b}$ where $h_2 \in G_m(k)$ such that $h_2^2 = h_1$, and so $h=h_2^4$.
\qed\end{proof}

\begin{lem}
\label{fixs}
$\Inv(J)^s \cong G_m(k) \times \Spin(10,k)/ \mu_4(k)$ 
\end{lem}

\begin{proof}
The fixed point group of $\Aut^+(B,-)^{\hat{s}}$ consists of the subgroup of $\Inv(J)$ that leaves $\beth$ invariant.  Recall the automorphisms that leave $\beth$ invariant are isomorphic to $\Spin(9,k)$ \citep{hu14,sv00} as the maps that preserve a quadratic form when restricted to an $11$ dimensional quadratic subalgebra of $J$.  Considering elements of $\Inv(J)$ instead we no longer need to leave the identity element of $J$ fixed, and so $\Spin(10,k) \subset \Inv(J)^{s}$.  There are also maps of the form $\nu_g$, from the previous lemma
\[ \nu_g \begin{bmatrix}
\xi_1 & c & \overline{b} \\
 \overline{c} & \xi_2 & a \\
b & \overline{a} & \xi_3
\end{bmatrix} =
\begin{bmatrix}
g^{-4}\xi_1 & g^{-1}c & g^{-1}\overline{b} \\
g^{-1}\overline{c} & g^2\xi_2 & g^2a \\
g^{-1}b & g^2\overline{a} & g^2\xi_3
\end{bmatrix}, \]
when $g\in G_m(k)$.  Notice $\nu_g \in \Spin(10,k)$ if and only if $g^4 = 1$.  Take 
\[ \nu: G_m(k) \times \Spin(10,k) \to \Inv(J)^{\hat{s}}, \]
to be the map $\nu(g,\phi) = \nu_g \phi$.  If we take $u$ as in \ref{k^4}
\[   u = \begin{bmatrix}
1 & \cdot &  \cdot \\
  \cdot &  \cdot &  \cdot \\
 \cdot &  \cdot &  \cdot
\end{bmatrix}, \]
then $\Spin(10,k)$ can be characterized as the elements in $\Inv(J)$ leaving $u$ fixed.  If $\varphi \in \Inv(J)^s$ then either $\varphi$ leaves $u$ fixed and $\varphi \in \Spin(10,k)$ or $\varphi(u) = g^{-4} u$ for some $g \in G_m(k)$ by \ref{k^4}.  If $\varphi(u) = g^{-4} u$ then $\nu_g^{-1} \phi \in \Spin(10,k)$ and so $\varphi = \nu_g \phi$ for some $\phi \in \Spin(10,k)$.  As we noted before $\ker(\nu) = \mu_4(k)$.
\qed\end{proof}

\begin{lem}
\label{fixtw}
$\Aut^+(B,-)^{\widehat{t_{\q}}\varpi} \cong\left( \Sp(4,D_{\q})/\mathbb{Z}_2\right)\rtimes \mathbb{Z}_2$
\end{lem}

\begin{proof}
Let
\[ j = \begin{bmatrix}
\xi_1 & c & \overline{b} \\
 \overline{c} & \xi_2 & a \\
b & \overline{a} & \xi_3
\end{bmatrix} \in J, \]
and
\[ \hat{j} = \begin{bmatrix}
			\alpha & j \\
			t(j) & \alpha
			\end{bmatrix} \in B^{\widehat{t_{\q}}\varpi}. \]
There is a multiplicative specialization, defined in \citep{ja61}, $\varphi: B^{\widehat{t_{\q}}\varpi} \to M_4(D_{\q},\iota_{\q})$ where $D_{\q}$ is the quaternion subalgebra of $C$ left invariant by $t_{\q}$ and $\iota_{\q}$ is the algebra involution on $M_4(D_{\q})$ induced by $t_{\q}$.  The map $\phi$ is given by
\[ \varphi(j) = \begin{bmatrix}
f_0 & a_1 & b_1 & c_1 \\  
\overline{a_1} & f_1 & c_0 & \overline{b_0} \\
\overline{b_1} & \overline{c_0} & f_2 & a_0 \\
\overline{c_1} & b_0 & \overline{a_0} & f_3
\end{bmatrix} \]
where
\begin{align*}
f_0 &= \frac{1}{2}(\alpha + j_1 + j_2 + j_3) \\
f_1 &= \frac{1}{2}(\alpha +j_1 - j_2 - j_3) \\
f_2 &= \frac{1}{2}(\alpha - j_1 + j_2 - j_3) \\
f_3 &= \frac{1}{2}(\alpha - j_1 - j_2 + j_3)
\end{align*}
Notice $\ker(\varphi) = \{\id, -\id\}$.  Now we have $\Aut^+(B,-)^{t_{\q}\varpi} \cong \Aut\big(M_4(D_{\q},\iota_{\q})\big) \cong (\Sp(4,D_{\q})/\mathbb{Z}_2) \rtimes \mathbb{Z}_2$.
\qed\end{proof}

\begin{prop}
The groups $\Aut^+(B,-)^{\varpi}$ and $\Aut^+(B,-)^{\hat{t}\varpi}$ are not isomorphic.
\end{prop}

\begin{proof}
$\Aut^+(B,-)^{\varpi}$ is of type $\F_4$, and $\Aut^+(B,-)^{\hat{t}\varpi}$ is of type $\C_4$. \qed
\qed\end{proof} 

By \ref{order2fixlem1} and \ref{order2fixlem2} we now have the following.

\begin{cor}
The are no isomorphisms between $B^{\varpi}$ and $B^{t \varpi}$.
\end{cor}

\begin{lem}
$\Aut^+(B,-)^{\widehat{t_{\q}}} \cong \Aut^+(\Her_3(D_{\q})) \times \Sp(1,D_{\q})/\mathbb{Z}_2$
\end{lem}

\begin{proof}
By \ref{fixedlemma} we have that $\Aut^+(B,-)^{\widehat{t_{\q}}}$ should be the subgroup of $\Aut^+(B,-)$ that leaves $B^{\widehat{t_{\q}}}$ invariant.  Clearly $\Aut^+(B,-)^{\widehat{t_{\q}}} \supset \Aut^+(\Her_3(D_{\q}))$.  By \citep{kr79} $\hat{t}$ represents a new class of invariant of type (1) and Table $2$ in \citep{kn14} and \citep{se06}, as explained in section $7$ below, the only type of spherical subgroup remaining is of type $\A_5 \times \A_1$.  Also, notice that $\Aut^+(B,-)^{\widehat{t_{\q}}} \supset \Aut(J)^{t_{\q}}$ and so contains $\Sp(1,D_{\q})$.  The action of $\Sp(1,D_{\q})$ is the extension of its action on $\Aut(C)$ in $\Aut(C)^{t_{\q}}$, which is left multiplication by $D_{\q}$ on $D_{\q}^{\perp}$.  By $\mathbb{Z}_2$ we mean $\{\id, -\id\} \subset \Aut^+(\Her_3(D_{\q}))$, which of type $\A_5$.
\qed\end{proof}

See \citep{yo90} for an explicit $k$-linear isomorphism.  By the previous results in this section we have the following collection of fixed point groups.

\begin{thm}
The following are the fixed point groups corresponding to isomorphism classes of $k$-involutions of $\Aut^+(B,-)$.
\begin{enumerate}[$1.$]
\item $\Aut^+(B,-)^{\hat{s}} \cong \big( G_m(k) \Spin(10,k)/\mu_4(k)\big) \rtimes \mathbb{Z}_2$ 
\item $\Aut^+(B,-)^{\widehat{t_{\q}}} \cong \Aut^+(\Her_3(D_{\q})) \times \Sp(1,D_{\q}))/\mathbb{Z}_2$ 
\item $\Aut^+(B,-)^{\varpi} \cong \Aut(J) \rtimes \mathbb{Z}_2$
\item $\Aut^+(B,-)^{\widehat{t_{\q}}\varpi} \cong \left(\Sp(4,D_{\q}))/\mathbb{Z}_2\right) \rtimes \mathbb{Z}_2$.
\end{enumerate}
\end{thm}

We see the following correspondence between $(\phi, Gal_k)$-diagrams of \citep[Table 1]{he00} and our representatives of invariants of type (1); $\hat{t}\varpi \leftrightarrow \E_{6,6}^0(I)$$\hat{t} \leftrightarrow \E_{6,6}^0(II)$, $\hat{s} \leftrightarrow \E_{6,6}^0(III)$, and $\varpi \leftrightarrow \E_{6,6}^0(IV)$.

\section{$\mu$-elements, Galois cohomology, and $\mathbb{Z}_2$-gradings}

Kac coordinates are described by V. Kac in \citep{ka81} for algebraically closed fields of characteristic zero and then generalized to positive characteristic by Serre in \citep{se06}.  One application of these coordinates is that they allow us to compute the Dynkin diagrams of the centralizers of elements of finite order in $\Aut(G)$ where $G$ is a quasisimple algebraic group.  This is done by fixing a base $\{\rho_i\}_{i\in I}$ for our root space $\Phi(T)$ of a maximal torus $T$, and writing the longest root $\tilde{\rho}$ is terms of this base,
\[ \tilde{\rho} = \sum_{i \in I} n_i \rho_i. \]
We can define $\rho_0 = -\tilde{\rho}$ with $I_0 = \{0\} \cup I$ and setting $n_0 = 1$ to obtain
\[ 0 = \sum_{i \in I_0} n_i \rho_i. \]
Define $\mathcal{M}$ to be the set of elements $x=(x_i)$ for $x_i \in \mathbb{Q}$ and $x_i \geq 0$ such that
\[ \sum_{i \in I_0} n_i x_i =1. \]
In \citep{se06} we see that any element of finite order in $G(k)$ is conjugate to an element 
\[ \theta_x \in \Hom\left(\lim_{\leftarrow} \mu_n, G \right), \]
where $x = (x_i) \in \mathcal{M}$.  Taking $x_i=\frac{s_i}{m} $ with $m, s_i \in \mathbb{N}$ and $\gcd(s_i) = 1$ we have
\[ \sum_{i \in I_0} n_i s_i =m. \]
If we let $m$ be the order of some $\theta_x$ and $G^x$ be the centralizer of $\theta_x$ in $G$, then the Dynkin diagram of $G^x$ has vertices in the extended Dynkin diagram corresponding to $G$ for each $i \in I_0$ such that $x_i = 0$.

We will consider the quasisimple algebraic group $\Inv(J)$ of type $\E_6$.  We first consider the extended Dynkin diagram of type $\E_6$.

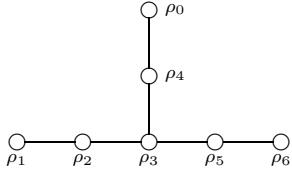
\begin{center}
\begin{picture}(120,30)(10,30)
\put(0,0){\circle{6}}
\put(25,0){\circle{6}}
\put(50,0){\circle{6}}
\put(75,0){\circle{6}}
\put(100,0){\circle{6}}
\put(0,-9){\makebox(0,0)[b]{\scriptsize $\rho_1$}}
\put(25,-9){\makebox(0,0)[b]{\scriptsize $\rho_2$}}
\put(50,-9){\makebox(0,0)[b]{\scriptsize $\rho_3$}}
\put(75,-9){\makebox(0,0)[b]{\scriptsize $\rho_5$}}
\put(100,-9){\makebox(0,0)[b]{\scriptsize $\rho_6$}}
\put(3,0){\line(1,0){19}}
\put(28,0){\line(1,0){19}}
\put(53,0){\line(1,0){19}}
\put(78,0){\line(1,0){19}}
\put(50,3){\line(0,1){19}}
\put(50,28){\line(0,1){19}}
\put(50,25){\circle{6}}
\put(50,50){\circle{6}}
\put(60,23){\makebox(0,0)[b]{\scriptsize $\rho_4$}}
\put(60,48){\makebox(0,0)[b]{\scriptsize $\rho_0$}}
\end{picture}
\end{center}

\vspace{1.5cm}

Using the numbering above, the longest root takes the form
\[ \rho_1 + 2\rho_2 + 3 \rho_3 + 2 \rho_4 + 2 \rho_5 + \rho_6. \]
Now we can use the rule described above and set 
\[ s_0 + s_1 + 2 s_2 + 3 s_3 + 2 s_4 + 2 s_5 + s_6 = 2, \]
and find the solutions over $\mathbb{N}$.

Considering the tori fixed by an inner automorphism that squares to two we get the diagrams for $\E_6$, $\D_5$, and $\A_5 \times \A_1$.  Corresponding to isomorphisms classes of type $\hat{t}$ are $\A_5 \times \A_1$ and are given for example by the solution $(0,0,1,0,0,0,0)$ which leaves the following diagram,

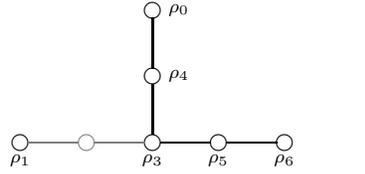
\begin{center}
\begin{picture}(120,30)(10,30)
\put(0,0){\circle{6}}
\put(25,0){\color{gray}\circle{6}}
\put(50,0){\circle{6}}
\put(75,0){\circle{6}}
\put(100,0){\circle{6}}
\put(0,-9){\makebox(0,0)[b]{\scriptsize $\rho_1$}}
\put(50,-9){\makebox(0,0)[b]{\scriptsize $\rho_3$}}
\put(75,-9){\makebox(0,0)[b]{\scriptsize $\rho_5$}}
\put(100,-9){\makebox(0,0)[b]{\scriptsize $\rho_6$}}
\put(3,0){\color{gray}\line(1,0){19}}
\put(28,0){\color{gray}\line(1,0){19}}
\put(53,0){\line(1,0){19}}
\put(78,0){\line(1,0){19}}
\put(50,3){\line(0,1){19}}
\put(50,28){\line(0,1){19}}
\put(50,25){\circle{6}}
\put(50,50){\circle{6}}
\put(60,23){\makebox(0,0)[b]{\scriptsize $\rho_4$}}
\put(60,48){\makebox(0,0)[b]{\scriptsize $\rho_0$}}
\end{picture},
\end{center}

\vspace{1.5cm}

and the diagram for type $\hat{s}$ has centralizer of type $\D_5$ and corresponds to the solution $(1,1,0,0,0,0,0)$ and thus the diagram

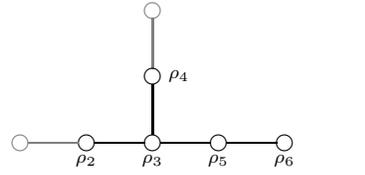
\begin{center}
\begin{picture}(120,30)(10,30)
\put(0,0){\color{gray}\circle{6}}
\put(25,0){\circle{6}}
\put(50,0){\circle{6}}
\put(75,0){\circle{6}}
\put(100,0){\circle{6}}
\put(25,-9){\makebox(0,0)[b]{\scriptsize $\rho_2$}}
\put(50,-9){\makebox(0,0)[b]{\scriptsize $\rho_3$}}
\put(75,-9){\makebox(0,0)[b]{\scriptsize $\rho_5$}}
\put(100,-9){\makebox(0,0)[b]{\scriptsize $\rho_6$}}
\put(3,0){\color{gray}\line(1,0){19}}
\put(28,0){\line(1,0){19}}
\put(53,0){\line(1,0){19}}
\put(78,0){\line(1,0){19}}
\put(50,3){\line(0,1){19}}
\put(50,28){\color{gray}\line(0,1){19}}
\put(50,25){\circle{6}}
\put(50,50){\color{gray}\circle{6}}
\put(60,23){\makebox(0,0)[b]{\scriptsize $\rho_4$}}
\end{picture}.
\end{center}

\vspace{1.5cm}

For the outer automorphisms we consider the twisted extended Dynkin diagram of type $\E_6^{(2)}$

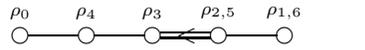
\begin{center}
\begin{picture}(120,24)(10,10)
\put(0,0){\circle{6}}
\put(25,0){\circle{6}}
\put(50,0){\circle{6}}
\put(75,0){\circle{6}}
\put(100,0){\circle{6}}
\put(0,6){\makebox(0,0)[b]{\scriptsize $\rho_0$}}
\put(25,6){\makebox(0,0)[b]{\scriptsize $\rho_4$}}
\put(50,6){\makebox(0,0)[b]{\scriptsize $\rho_3$}}
\put(75,6){\makebox(0,0)[b]{\scriptsize $\rho_{2,5}$}}
\put(100,6){\makebox(0,0)[b]{\scriptsize $\rho_{1,6}$}}
\put(3,0){\line(1,0){19}}
\put(28,0){\line(1,0){19}}
\put(53,-1){\line(1,0){19}}
\put(53,1){\line(1,0){19}}
\put(59,-2.5){$<$}
\put(78,0){\line(1,0){19}}
\end{picture}.
\end{center}

\vspace{1.5cm}

The the solution $(2,0,0,0,0,0,0)$ corresponds to $\varpi$, which has centralizer of type $\F_4$ and diagram 

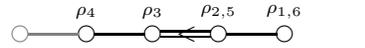
\begin{center}
\begin{picture}(120,24)(10,10)
\put(0,0){\color{gray}\circle{6}}
\put(25,0){\circle{6}}
\put(50,0){\circle{6}}
\put(75,0){\circle{6}}
\put(100,0){\circle{6}}
\put(25,6){\makebox(0,0)[b]{\scriptsize $\rho_4$}}
\put(50,6){\makebox(0,0)[b]{\scriptsize $\rho_3$}}
\put(75,6){\makebox(0,0)[b]{\scriptsize $\rho_{2,5}$}}
\put(100,6){\makebox(0,0)[b]{\scriptsize $\rho_{1,6}$}}
\put(3,0){\color{gray}\line(1,0){19}}
\put(28,0){\line(1,0){19}}
\put(53,-1){\line(1,0){19}}
\put(53,1){\line(1,0){19}}
\put(59,-2.5){$<$}
\put(78,0){\line(1,0){19}}
\end{picture}.
\end{center}

\vspace{1.5cm}

Lastly, the solution $(0,1,0,0,0,0,1)$ corresponds to $\hat{t}\varpi$ and a Dynkin diagram of type $\C_4$

\begin{center}
\begin{picture}(120,24)(10,10)
\put(0,0){\circle{6}}
\put(25,0){\circle{6}}
\put(50,0){\circle{6}}
\put(75,0){\circle{6}}
\put(100,0){\color{gray}\circle{6}}
\put(0,6){\makebox(0,0)[b]{\scriptsize $\rho_0$}}
\put(25,6){\makebox(0,0)[b]{\scriptsize $\rho_4$}}
\put(50,6){\makebox(0,0)[b]{\scriptsize $\rho_3$}}
\put(75,6){\makebox(0,0)[b]{\scriptsize $\rho_{2,5}$}}
\put(3,0){\line(1,0){19}}
\put(28,0){\line(1,0){19}}
\put(53,-1){\line(1,0){19}}
\put(53,1){\line(1,0){19}}
\put(59,-2.5){$<$}
\put(78,0){\color{gray}\line(1,0){19}}
\end{picture}.
\end{center}

\vspace{1.5cm}

To describe the Galois cohomology groups corresponding to symmetric $k$-varieties let $\mathcal{C}_k$ be the set of distinct isomorphism classes of $k$-involutions of $\Aut(G)$ where $G$ is a reductive linear algebraic group over a field of not characteristic $2$.  We let $Gal_k$ denote the absolute Galois group of the field $k$.  Let $A$ be a $k$-algebra, then as described in \citep{se97} the cohomology group $H^0(Gal_k,\Aut(A,K)) \cong \Aut(A,k)$.  If we then consider a subalgebra of $A \supset A^{r}$ fixed by an element of $[r] \in \mathcal{C}_k$, then $H^1(Gal_k,\Aut(A,A^r,K))$ corresponds to the $K/k$-forms of $A^r$ by \ref{order2fixlem1} and \ref{order2fixlem2}.  And by \ref{fixedlemma}, $H^1(Gal_k,\Aut(A,A^r,K))$ is in bijection with $H^1(Gal_k, Z_G(r))$.  A more detailed description of the Galois cohomology related to $k$-involutions of type $\G_2$ and $\F_4$ can be found in \citep{hu14,hu12}.  

When talking about automorphisms of order $2$ of algebraic groups it is natural to consider their correspondence to $\mathbb{Z}_2$-gradings of the corresponding Lie algebras.  These have been studied recently by Draper, Mart\'in, and Viruel for Lie algebras of type $\mathfrak{g}_2$ \citep{dm06}, $\mathfrak{f}_4$\citep{dm09}, and $\mathfrak{e}_6$ \citep{dv-}.   A \emph{grading} by a finite abelian group $G$ on an algebra $\mathcal{A}$ is a vector space decomposition
\[ \mathcal{A} = \bigoplus_{g\in G} \mathcal{A}_g, \]
satisfying $\mathcal{A}_g \mathcal{A}_h \subset \mathcal{A}_{gh}$ for all $g,h \in G$.

The following two results appear as Theorem 1.38 and 1.39 in \citep{em13}. 

\begin{thm}
Let $\mathcal{A}$ and $\mathcal{B}$ be finite-dimensional (nonassociative) algebras.  Assume $\theta: \Aut(\mathcal{A}) \to \Aut(\mathcal{B})$.  Then, for any abelian group $G$ we have a mapping of $G$-gradings, $\Gamma \to \theta(\Gamma)$, from $G$-gradings on $\mathcal{A}$ to $G$-gradings on $\mathcal{B}$.  If $\Gamma$ and $\Gamma'$ are isomorphic , then $\theta(\Gamma)$ and $\theta(\Gamma')$ are isomorphic.
\end{thm}

\begin{thm}
Assume we have an isomorphism $\theta:\Aut(\mathcal{A}) \to \Aut(\mathcal{B})$, then two fine abelian group gradings $\Gamma$ and $\Gamma'$ are equivalent if and only if $\theta(\Gamma)$ and $\theta(\Gamma')$ are equivalent.
\end{thm}

So we see in the cases when $\mathcal{A}$ is an eight dimensional composition algebra or an Albert algebra the abelian group gradings on $\Der(\mathcal{A})$ correspond to the gradings on $\mathcal{A}$, since $\Aut(\mathcal{A}) \cong \Aut(\Der(\mathcal{A}))$.  The $\E_6$ case is slightly different as the center is non-trivial in general.  Letting $\mathfrak{e}_6$ be a simple Lie algebra of type $\E_6$ (for $J$ an Albert algebra $\mathfrak{e}_6 = \Lie(\Inv(J))$ for $J$ an Albert algebra), then the identity component of $\Aut(\mathfrak{e}_6)$ is isomorphic to $\Inv(J)/Z(\Inv(J))$. This gives us the finite abelian gradings that correspond to the inner $k$-involutions, and the outer $k$-involutions correspond to the gradings induced by the symmetric pairs $(\E_6, \F_4)$ and $(\E_6,\C_4)$. All of these results and much more on gradings of $\mathfrak{e}_6$ can be found in \citep{dv-}.  

When considering these exceptional groups we can also make the connection to $\mathbb{Z}_2$-gradings of their respective defining representations; octonion algebras, Albert algebras, and Brown algebras.  This is often how gradings of simple Lie algebras ar obtained as well, for example \citep{dm06, dm09, el98, el09}.  The following result can be found as Proposition 4.8 in \citep{ko09} where $G^D$ is the Cartier dual of $G$.

\begin{prop}
The gradings on a $k$-algebra $\mathcal{A}$ by a finitely generated abelian group $G$ are in one-to-one correspondence with the embeddings of the algebraic group scheme $G^D$ into $\Aut(\mathcal{A})$.
\end{prop}

It is well known and straight forward to verify that $(\mathbb{Z}_n)^D \cong \mu_n$ as group schemes.  This gives us an immediate correspondence between inner $k$-involutions and $\mathbb{Z}_2$-gradings of these defining representations, since over an algebraically closed field of characteristic not $2$ inner automorphisms of order $2$ correspond with embeddings of $\mu_2$ into $\Aut(\mathcal{A})$ \citep{se06}.  The gradings of octonion algebras have been studied in \citep{el98} and Albert algebras in \citep{el09} and \citep{dm09}.  It is also worth noting that as we descend away from algebraic closure the tendency for conjugacy classes of elements of order $n$ to break up into more conjugacy classes is partially determined by the Galois cohomology of $\mu_2$.  When $k$ is a field these elements of finite order correspond to embeddings of $\mu_n(k)$, and so these obstructions correspond to
\[ H^1(k,\mu_n) \cong G_m(k)/G_m(k)^n, \]
when $\ch(k) \neq n$.

\end{document}